\documentclass{article}
\usepackage{amsmath, amssymb, amsfonts, amsthm, dsfont}
\usepackage{tikz}

\let\originalleft\left
\let\originalright\right
\renewcommand{\left}{\mathopen{}\mathclose\bgroup\originalleft}
\renewcommand{\right}{\aftergroup\egroup\originalright}

\newcommand{\ind}[1]{\mathds{1}_{\left\lbrace #1 \right\rbrace}}

\newcommand{\dd}{\mathrm{d}}

\theoremstyle{plain}

\newtheorem{theorem}{Theorem}[section]

\newtheorem{lemma}[theorem]{Lemma}

\theoremstyle{definition}

\title{On the upper bound in Varadhan's Lemma}
\author{H. M. Jansen$^{1,2}$, M. R. H. Mandjes$^{1}$, K. De Turck$^{2}$, S. Wittevrongel$^{2}$}

\begin{document}
\maketitle

\begin{abstract}
\noindent In this paper, we generalize the upper bound in Varadhan's Lemma. The standard formulation of Varadhan's Lemma contains two important elements, namely an upper semicontinuous integrand and a rate function with compact sublevel sets. However, motivated by results from queueing theory, we do not assume that rate functions have compact sublevel sets. Moreover, we drop the assumption that the integrand is upper semicontinuous and replace it by a weaker condition. We prove that the upper bound in Varadhan's Lemma still holds under these weaker conditions. Additionally, we show that only measurability of the integrand is required when the rate function is continuous.
\end{abstract}

\newline

\noindent $^{1}$ Korteweg-de Vries Institute for Mathematics,
University of Amsterdam, Science Park 904, 1098 XH Amsterdam, the Netherlands.
\newline

\noindent $^{2}$ TELIN, Ghent University, Sint-Pietersnieuwstraat 41,
B-9000 Ghent, Belgium.
\newline

\noindent {\it E-mail}. {\tt\{h.m.jansen|m.r.h.mandjes\}@uva.nl}, {\tt\{kdeturck|sw\}@telin.ugent.be}

\section{Introduction}
Exponential integrals often play an important role in the proof of a large deviations principle (LDP). Varadhan's Lemma is a powerful generalization of Laplace's method for computing exponential integrals. Especially the upper bound in Varadhan's Lemma turns out to be a very useful tool for proving LDPs. However, Varadhan's Lemma is stated under somewhat restrictive conditions, which rule out many interesting cases. In particular, certain rate functions arising in queueing theory do not satisfy the conditions of Varadhan's Lemma.
Motivated by this observation, we will generalize the upper bound in Varadhan's Lemma.

\section{Main result}
Let $\mathcal{X}$ be a topological space and denote its Borel $\sigma$-algebra by $\mathcal{B}$. Throughout, we will assume that $\left\lbrace \mu_{n} \right\rbrace_{n \in \mathbb{N}}$ is a sequence of probability measures defined on $\mathcal{B}$. We will say that the sequence $\left\lbrace \mu_{n} \right\rbrace_{n \in \mathbb{N}}$ satisfies an LDP with rate function $J$ if
\begin{align*}
\limsup_{n \to \infty} \frac{1}{n} \log \mu_{n} \left( F \right) &\leq - \inf_{x \in F} J \left( x \right)
\intertext{for any closed set $F \subset \mathcal{X}$ and}
\limsup_{n \to \infty} \frac{1}{n} \log \mu_{n} \left( G \right) &\geq - \inf_{x \in G} J \left( x \right)
\end{align*}
for any open set $G \subset \mathcal{X}$, where $J \colon \mathcal{X} \to \left[ 0,\infty \right]$ is a lower semicontinuous function. Note that we do not assume that $J$ has compact sublevel sets, i.e., we do not assume that $J$ is a good rate function.

An important goal of this paper is to prove the following lemma. Note that this is just the upper bound in Varadhan's Lemma, but without the assumption that $J$ is a good rate function. Moreover, the lemma states that a well known tail condition is both necessary and sufficient for the upper bound to hold. Although this is not very surprising, it is never explicitly stated like this.
\begin{lemma}
Suppose that the sequence of measures $\left\lbrace \mu_{n} \right\rbrace_{n \in \mathbb{N}}$ satisfies an LDP with rate function $J$ and let $\phi \colon \mathcal{X} \to \mathbb{R}$ be an upper semicontinuous function. Then it holds that
\begin{align*}
\limsup_{n \to \infty} \frac{1}{n} \log \int_{\mathcal{X}} e^{n \phi \left( x \right)} \mu_{n} \left( \dd x \right) \leq \sup_{x \in \mathcal{X}} \left[ \phi \left( x \right) - J \left( x \right) \right]
\end{align*}
if and only if
\begin{align*}
\lim_{M \to \infty} \limsup_{n \to \infty} \frac{1}{n} \log \int_{\mathcal{X}} e^{n \phi \left( x \right)} \ind{\phi \left( x \right) > M} \mu_{n} \left( \dd x \right) \leq \sup_{x \in \mathcal{X}} \left[ \phi \left( x \right) - J \left( x \right) \right].
\end{align*}
\end{lemma}
This lemma is an immediate result from the following more general lemma, which is the main result of this paper. 
Its proof is inspired by the proof of Varadhan's Lemma given in \cite{denhollander2000}.
As is customary, we define $\exp \left( - \infty \right) = 0$, $\log \left( 0 \right) = -\infty$ and $\exp \left( \infty \right) = \log \left( \infty \right) = \infty$. Throughout, we will denote the closure of a set $A$ by $\mathrm{cl} A$.
\begin{lemma}
Suppose that the sequence of measures $\left\lbrace \mu_{n} \right\rbrace_{n \in \mathbb{N}}$ satisfies an LDP with rate function $J$. Let $\phi \colon \mathcal{X} \to \left[ -\infty , \infty \right]$ be a Borel measurable function and define $\phi_{M} = \phi \wedge M$ for $M \in \mathbb{R}$. Assume that at least one of the following conditions is true:
\begin{enumerate}
\item $J$ is continuous;
\item the superlevel set $\phi^{-1} \left( \left[ w,\infty \right] \right)$ is closed for every $w \in \mathbb{R}$ satisfying the inequality $w \geq \lim_{M \to \infty} \sup_{x \in \mathcal{X}} \left[ \phi_{M} \left( x \right) - J \left( x \right) \right]$.
\end{enumerate}
Then it holds that
\begin{align*}
\limsup_{n \to \infty} \frac{1}{n} \log \int_{\mathcal{X}} e^{n \phi \left( x \right)} \mu_{n} \left( \dd x \right) \leq \lim_{M \to \infty} \sup_{x \in \mathcal{X}} \left[ \phi_{M} \left( x \right) - J \left( x \right) \right]
\end{align*}
if and only if
\begin{align*}
\lim_{M \to \infty} \limsup_{n \to \infty} \frac{1}{n} \log \int_{\mathcal{X}} e^{n \phi \left( x \right)} \ind{\phi \left( x \right) > M} \mu_{n} \left( \dd x \right) \leq \lim_{M \to \infty} \sup_{x \in \mathcal{X}} \left[ \phi_{M} \left( x \right) - J \left( x \right) \right].
\end{align*}
\end{lemma}
\begin{proof}
For notational convenience, define $\beta_{M} = \sup_{x \in \mathcal{X}} \left[ \phi_{M} \left( x \right) - J \left( x \right) \right]$ for $M \in \mathbb{R}$. Note that $\beta_{M}$ is well defined for each $M \in \mathbb{R}$ and that $\beta_{M}$ is nondecreasing in $M$. Hence, $\lim_{M \to \infty} \beta_{M}$ is well defined.

The statement is obviously true if $\lim_{M \to \infty} \beta_{M} = \infty$, so in the remainder of this proof we will assume that $\lim_{M \to \infty} \beta_{M} < \infty$.

Fix any $b \in \mathbb{R}$ such that $b > \lim_{M \to \infty} \beta_{M}$ and pick any $w \in \left( -\infty,b \right]$ such that $w \geq \lim_{M \to \infty} \beta_{M}$. For $k \in \mathbb{N}$, define the measurable sets
\begin{align*}
L^{k}_{i} = \phi^{-1}_{b} \left( \left[ c^{k}_{i-1},c^{k}_{i} \right] \right)
\end{align*}
for $i = 1 , \dotsc , k$, where
\begin{align*}
c^{k}_{i} = w - \frac{i}{k} \left( w-b \right)
\end{align*}
for $i = 0 , \dotsc , k$. Observe that $c^{k}_{i} - c^{k}_{i-1} = - \frac{w-b}{k}$ and that $L = \phi^{-1}_{b} \left( \left[ w,b \right] \right) = \cup_{i=1}^{k} L^{k}_{i}$ for every $k \in \mathbb{N}$.

Obviously, it holds that
\begin{align*}
\limsup_{n \to \infty} \frac{1}{n} \log \int_{\mathcal{X}} e^{n \phi_{b} \left( x \right)} \mu_{n} \left( \dd x \right) &=\\
\limsup_{n \to \infty} \frac{1}{n} \log \left( \int_{L} e^{n \phi_{b} \left( x \right)} \mu_{n} \left( \dd x \right) + \int_{L^{\complement}} e^{n \phi_{b} \left( x \right)} \mu_{n} \left( \dd x \right) \right) &=\\
\max \left\lbrace \limsup_{n \to \infty} \frac{1}{n} \log \int_{L} e^{n \phi_{b} \left( x \right)} \mu_{n} \left( \dd x \right) , \limsup_{n \to \infty} \frac{1}{n} \log \int_{L^{\complement}} e^{n \phi_{b} \left( x \right)} \mu_{n} \left( \dd x \right) \right\rbrace
\end{align*}
and
\begin{align*}
\limsup_{n \to \infty} \frac{1}{n} \log \int_{L^{\complement}} e^{n \phi_{b} \left( x \right)} \mu_{n} \left( \dd x \right) \leq w.
\end{align*}
Now fix $k \in \mathbb{N}$. We have
\begin{align*}
\limsup_{n \to \infty} \frac{1}{n} \log \int_{L} e^{n \phi_{b} \left( x \right)} \mu_{n} \left( \dd x \right)
&\leq \limsup_{n \to \infty} \frac{1}{n} \log \sum_{i=1}^{k}\int_{L^{k}_{i}} e^{n \phi_{b} \left( x \right)} \mu_{n} \left( \dd x \right)\\
&= \max_{i = 1 , \dotsc , k} \limsup_{n \to \infty} \frac{1}{n} \log \int_{L^{k}_{i}} e^{n \phi_{b} \left( x \right)} \mu_{n} \left( \dd x \right)\\
&\leq \max_{i = 1 , \dotsc , k} \limsup_{n \to \infty} \frac{1}{n} \log \int_{L^{k}_{i}} e^{n c^{k}_{i}} \mu_{n} \left( \dd x \right).
\end{align*}
Observe that for $i = 1 , \dotsc , k$ it holds that
\begin{align*}
\limsup_{n \to \infty} \frac{1}{n} \log \int_{L^{k}_{i}} e^{n c^{k}_{i}} \mu_{n} \left( \dd x \right)
&= c^{k}_{i} + \limsup_{n \to \infty} \frac{1}{n} \log \mu_{n} \left( L^{k}_{i} \right)\\
&\leq c^{k}_{i} + \limsup_{n \to \infty} \frac{1}{n} \log \mu_{n} \left( \mathrm{cl} L^{k}_{i} \right)\\
&\leq c^{k}_{i} - \inf_{x \in \mathrm{cl} L^{k}_{i}} J \left( x \right)\\
&= \sup_{x \in \mathrm{cl} L^{k}_{i}} \left[ c^{k}_{i} - J \left( x \right) \right].
\end{align*}

Suppose that the first condition is true. Then
\begin{align*}
\sup_{x \in \mathrm{cl} L^{k}_{i}} \left[ c^{k}_{i} - J \left( x \right) \right] = \sup_{x \in L^{k}_{i}} \left[ c^{k}_{i} - J \left( x \right) \right],
\end{align*}
by continuity of $J$. But $c^{k}_{i-1} = c^{k}_{i} + \frac{w-b}{k}$, so $c^{k}_{i} \leq \phi_{b} \left( x \right) - \frac{w-b}{k}$ for all $x \in L^{k}_{i}$. Hence, we get
\begin{align*}
\sup_{x \in L^{k}_{i}} \left[ c^{k}_{i} - J \left( x \right) \right]
&\leq \sup_{x \in L^{k}_{i}} \left[ \phi_{b} \left( x \right) - \frac{w-M}{k} - J \left( x \right) \right]\\
&\leq \sup_{x \in \mathcal{X}} \left[ \phi_{b} \left( x \right) - J \left( x \right) \right] - \frac{w-b}{k}.
\end{align*}

Suppose that the second condition is true. Then we have $\mathrm{cl} L^{k}_{i} \subset \mathrm{cl} \phi^{-1}_{b} \left( \left[ c^{k}_{i-1},b \right] \right)$ and $\phi^{-1}_{b} \left( \left[ c^{k}_{i-1},b \right] \right) = \phi^{-1} \left( \left[ c^{k}_{i-1},\infty \right] \right)$ is closed by assumption, so $\mathrm{cl} L^{k}_{i} \subset \phi^{-1}_{b} \left( \left[ c^{k}_{i-1},b \right] \right)$. We get $c^{k}_{i} \leq \phi_{b} \left( x \right) - \frac{w-b}{k}$ for all $x \in \mathrm{cl} L^{k}_{i}$ and
\begin{align*}
\sup_{x \in \mathrm{cl} L^{k}_{i}} \left[ c^{k}_{i} - J \left( x \right) \right]
&\leq \sup_{x \in \mathrm{cl} L^{k}_{i}} \left[ \phi_{b} \left( x \right) - \frac{w-M}{k} - J \left( x \right) \right]\\
&\leq \sup_{x \in \mathcal{X}} \left[ \phi_{b} \left( x \right) - J \left( x \right) \right] - \frac{w-b}{k}.
\end{align*}

Note that it does not matter which of the two conditions is true: we get the same inequality in both cases. Consequently, for every $k \in \mathbb{N}$ it holds that
\begin{align*}
\limsup_{n \to \infty} \frac{1}{n} \log \int_{L} e^{n \phi_{b} \left( x \right)} \mu_{n} \left( \dd x \right)
&\leq \max_{i = 1 , \dotsc , k} \limsup_{n \to \infty} \frac{1}{n} \log \int_{L^{k}_{i}} e^{n c^{k}_{i}} \mu_{n} \left( \dd x \right)\\
&\leq \sup_{x \in \mathcal{X}} \left[ \phi_{b} \left( x \right) - J \left( x \right) \right] - \frac{w-b}{k}\\
&\leq w - \frac{w-b}{k},
\end{align*}
so
\begin{align*}
\limsup_{n \to \infty} \frac{1}{n} \log \int_{L} e^{n \phi_{b} \left( x \right)} \mu_{n} \left( \dd x \right) &\leq w.
\end{align*}
Combining all results, we obtain
\begin{align*}
\limsup_{n \to \infty} \frac{1}{n} \log \int_{\mathcal{X}} e^{n \phi_{b} \left( x \right)} \mu_{n} \left( \dd x \right) &=\\
\max \left\lbrace \limsup_{n \to \infty} \frac{1}{n} \log \int_{L} e^{n \phi_{b} \left( x \right)} \mu_{n} \left( \dd x \right) , \limsup_{n \to \infty} \frac{1}{n} \log \int_{L^{\complement}} e^{n \phi_{b} \left( x \right)} \mu_{n} \left( \dd x \right) \right\rbrace &\leq w.
\end{align*}
Because this holds for all $w \in \left( -\infty,b \right]$ with $w \geq \lim_{M \to \infty} \beta_{M}$, it follows immediately that
\begin{align*}
\limsup_{n \to \infty} \frac{1}{n} \log \int_{\mathcal{X}} e^{n \phi_{b} \left( x \right)} \mu_{n} \left( \dd x \right) \leq \lim_{M \to \infty} \sup_{x \in \mathcal{X}} \left[ \phi_{M} \left( x \right) - J \left( x \right) \right]
\end{align*}
for all $b \in \mathbb{R}$.

Now observe that for each $b \in \mathbb{R}$ we have
\begin{align*}
\limsup_{n \to \infty} \frac{1}{n} \log \int_{\mathcal{X}} e^{n \phi \left( x \right)} \mu_{n} \left( \dd x \right) &\leq\\
\limsup_{n \to \infty} \frac{1}{n} \log \left[ \int_{\mathcal{X}} e^{n \phi_{b} \left( x \right)} \mu_{n} \left( \dd x \right) 
+ \int_{\mathcal{X}} e^{n \phi \left( x \right)} \ind{\phi \left( x \right) > b} \mu_{n} \left( \dd x \right) \right] &=\\
\max \left\lbrace \limsup_{n \to \infty} \frac{1}{n} \log \int_{\mathcal{X}} e^{n \phi_{b} \left( x \right)} \mu_{n} \left( \dd x \right) , \limsup_{n \to \infty} \frac{1}{n} \int_{\mathcal{X}} e^{n \phi \left( x \right)} \ind{\phi \left( x \right) > b} \mu_{n} \left( \dd x \right) \right\rbrace &\leq\\
\max \left\lbrace \lim_{M \to \infty} \sup_{x \in \mathcal{X}} \left[ \phi_{M} \left( x \right) - J \left( x \right) \right]
, \limsup_{n \to \infty} \frac{1}{n} \int_{\mathcal{X}} e^{n \phi \left( x \right)} \ind{\phi \left( x \right) > b} \mu_{n} \left( \dd x \right) \right\rbrace,
\end{align*}
which implies the statement of the lemma.
\end{proof}

\bibliographystyle{plain}
\bibliography{ref}

\end{document}